\documentclass{amsart}
\usepackage{amsmath,graphics}
\usepackage{amsfonts,amssymb}

\usepackage{ulem}

\usepackage[latin1]{inputenc}

\usepackage[T1]{fontenc}

\usepackage{eucal,hyperref} 
\usepackage{cmmib57} 

\theoremstyle{plain}
\newtheorem{theorem}{Theorem}[section]
\newtheorem{lemma}[theorem]{Lemma}

\newtheorem{proposition}[theorem]{Proposition}

\theoremstyle{remark}
\newtheorem*{remark}{Remark}
\newtheorem*{definition}{Definition}

\theoremstyle{definition}

\newcommand{\RC}{{\mathcal R}}

\newcommand{\Z}{\mathbb{Z}}
\newcommand{\N}{\mathbb{N}}

\newcommand{\R}{\mathbb{R}}
\newcommand{\C}{\mathbb{C}}

\def\cp{\mathbb{CP}}

\def\part{\partial}
\def\sm{\setminus}

\def\Hom{\operatorname{Hom}}

\def\im{\mbox{Im}}

\def\wti{\widetilde}

\def\lra{\longrightarrow}

\def\bd{\begin{definition}}
\def\ed{\end{definition}}
\def\bt{\begin{theorem}}
\def\et{\end{theorem}}
\def\br{\begin{remark}}
\def\er{\end{remark}}
\def\bp{\begin{proposition}}
\def\ep{\end{proposition}}
\def\be{\begin{equation}}
\def\ee{\end{equation}}
\def\bn{\begin{enumerate}}
\def\en{\end{enumerate}}
\def\ba{\begin{array}}
\def\ea{\end{array}}
\def\a{\alpha}
\def\op{\operatorname}
\def\gl{\op{GL}}
\def\wti{\widetilde}
\def\G{\Gamma}
\def\id{\op{id}}
\begin{document}

\title[Novikov homology]{Twisted Novikov homology of complex hypersurface complements}
\author{Stefan Friedl}
\address{Department of Mathematics,
          University of Regensburg,
          Germany.}
\email {sfriedl@gmail.com}
\author{Laurentiu Maxim}
\address{Department of Mathematics,
          University of Wisconsin-Madison, USA.}
\email {maxim@math.wisc.edu}


\date{\today}



\begin{abstract}
We study the twisted Novikov homology of the complement of a complex hypersurface in general position at infinity. We give a self-contained topological proof of the vanishing (except possibly in the middle degree) of the twisted Novikov homology groups associated to 
positive cohomology classes of degree one defined on the complement.
\end{abstract}

\maketitle

\section{Introduction}

Novikov homology was originally introduced for the purpose of generalizing classical Morse theory to the context of arbitrary closed one-forms, e.g., see \cite{F04} for an overview.  This theory and its variants have fascinating applications in dynamical systems, symplectic topology, geometry group theory, knot theory, etc. For example, twisted Novikov homology can detect fibering of knots. In relation to geometric group theory, the so-called BNSR-invariants of a group, which contain important information on the finiteness properties of certain subgroups, can be described in terms of vanishing results in Novikov homology (e.g., see \cite{FGS,PS10}). Other topological implications of vanishing of Novikov homology have been derived through the language of barcodes and Jordan cells (e.g., see \cite{B}).

In this note, we study the twisted Novikov homology of complements to complex hypersurfaces in general position at infinity. We give a self-contained topological proof of the vanishing (except possibly in the middle degree) of the twisted Novikov homology groups associated to 
positive cohomology classes of degree one defined on the complement. Classical Novikov homology of (essential) hyperplane arrangement complements has been studied in \cite{KP15} by Morse-theoretical methods.

\medskip

Let $M$ be a  topological space. Throughout the paper we assume that all topological spaces are connected and that they admit universal coverings. Furthermore  we make the canonical identifications $H^1(M;\R)=\Hom(H_1(M,\Z),\R)=\Hom(\pi_1(M),\R)$.
An {\it{admissible pair}} for $M$ is an epimomorphism $\psi\colon \pi_1(M)\to \Gamma$ to a free abelian group together with some $\xi \in H^1(M;\R)=\Hom(\pi_1(M),\R)$ such that $\xi\colon \pi_1(M)\to \R$ factors through $\psi$. By a slight abuse of notation we denote the unique induced homomorphism $\Gamma\to \R$ by $\xi$ as well. 

In Section~\ref{defNB}, we will associate to $(M,\psi,\xi)$, with $(\psi,\xi)$ an admissible pair together with a representation $\alpha\colon \pi_1(M)\to \gl(k,S)$ over a domain $S$, the {\it twisted Novikov-Betti numbers} $b_i^\a(M,\psi,\xi)$ and the {\it twisted Novikov torsion numbers} $q_i^\a(M,\psi,\xi)$. The classical Novikov-Betti numbers $b_i(M,\xi)$ and Novikov torsion numbers $q_i(M,\xi)$  as defined in \cite{F04} 
can be recovered by taking $\alpha$ to be the trivial one-dimensional representation $\pi_1(M)\to \gl(1,\Z)$ and by taking  $\psi=\xi\colon \pi_1(M)\to \Gamma:=\mbox{Im}(\xi)\subset \R$. 

Before we state our main theorem we recall that for a complex hypersurface $X\subset \C^n$ the first homology of the complement $M_X$ has a basis that is given by the choice of a positive meridian for each irreducible component of $X$. We say that a homomorphism $\xi\colon H_1(M_X;\R)\to \R$ is {\it positive}, if $\xi$ maps all meridians to positive real numbers.

\bt\label{mt}\label{mainthm} Let $X \subset \C^n$ be a complex hypersurface in general position at infinity, with complement $M_X$. Then for any admissible pair $(\psi\colon \pi_1(M_X)\to \Gamma,\xi \in H^1(M_X;\R))$ such that $\xi$ is positive, together with a representation $\alpha\colon \pi_1(M_X)\to \gl(k,S)$ over a domain $S$, we have
\be\label{B} b_i^\a(M_X,\psi,\xi)= \left\{ \ba{rl} 0, & \mbox{ for }i\ne n, \\ (-1)^{n}k\chi(M_X),
&\mbox{ for }i=n. \ea \right.\ee
In particular, \be\label{E}(-1)^n \cdot \chi(M_X) \geq 0.\ee
Moreover, all twisted Novikov torsion numbers of $(M_X,\psi,\xi)$ vanish, that is:
\be\label{tor} q_i^\a(M_X,\psi,\xi)=0, \  {\rm for \ all \ } i \geq 0.\ee
\et

\subsection*{Conventions.} All domains are understood to be commutative.

\subsection*{Acknowledgement.}
S. Friedl gratefully acknowledges the support provided by the SFB 1085 `Higher
Invariants' at the University of Regensburg, funded by the Deutsche
For\-schungsgemeinschaft (DFG). L. Maxim is supported by grants from NSF, NSA, and by a fellowship from the Max-Planck-Institut f\"ur Mathematik,  Bonn.

\section{Preliminaries}\label{pre}

\subsection{Novikov rings}

Let $\Gamma$ be a free abelian group and let $S$ be a domain.
\begin{enumerate}
\item We say that $p\in  S[\Gamma]$ is a {\it{monomial}} if there exists $\gamma\in \Gamma$ and a unit $s\in S$  such that $p=s\gamma$.
\item Given $\xi \in \Hom(\Gamma,\R)$ and $p=\sum_{\gamma\in \Gamma}n_\gamma \gamma\in S[\Gamma]\setminus \{0\}$  we write 
\[ m_\xi(p):=\max\{ \xi(\gamma)\,|\,n_\gamma\ne 0\}\]
and we write
\[ t_\xi(p):=\sum_{\xi(\gamma)=m_\xi(p)} n_\gamma\gamma.\]
\item 
Given $\xi \in \Hom(\Gamma,\R)$ we write 
\[ T_\xi S[\Gamma]:=\{ p\in S[\Gamma]\setminus \{0\}\,|\,\mbox{$t_\xi(\gamma)$ is a monomial}\}.\]
Furthermore we write 
\[ \RC_\xi S[\Gamma]:=(T_\xi S[ \Gamma])^{-1}S[\Gamma].\]
We refer to $\RC_\xi S[\Gamma]$ as the {\it rational Novikov completion of $S[\Gamma]$ with respect to $\xi$}.
\end{enumerate}

We recall the following well-known lemma.

\begin{lemma}\label{lem:units-in-rc}
Let $\Gamma$ be a free abelian group and let $S$ be a domain.
Then $T_\xi S[\Gamma]$ consists precisely of the elements of $S[\Gamma]$ which are invertible in $\RC_\xi S[\Gamma]$.
\end{lemma}

\begin{proof}
Let $p\in S[\Gamma]$ be an element that is invertible in $\RC_\xi S[\Gamma]$. This means that there there exists a $r\in T_\xi \Gamma$ and a $q\in S[\Gamma]$ such that $p\cdot r^{-1}q=1$. In particular $p\cdot q=r$. Since $S$ is a domain  it follows  that $t_\xi \colon S[\Gamma]\setminus \{0\}\to S[\Gamma]\setminus \{0\}$ is multiplicative. We now see that
\[  t_\xi(p)\cdot t_\xi(q)\,=\,t_\xi(r)\,=\,1.\]
It follows that $t_\xi(p)$ is a unit in $S[\Gamma]$. Since $S$ is a domain  we deduce that $t_\xi(p)$ is a monomial, i.e.\ $p\in T_\xi S[\Gamma]$. 
\end{proof}

%
%

\subsection{Novikov-Betti and torsion numbers}\label{defNB}
Let $M$ be a topological space. We write $\pi=\pi_1(M)$. Let $(\psi\colon \pi\to \Gamma,\xi\in H^1(M;\R))$ be an admissible pair
and let
$\alpha\colon \pi_1(M)\to \gl(k,S)$ be a representation over a domain $S$.
 We denote by $$\wti{M} \lra M$$  the universal covering of $M$.
The canonical left $\pi$-action on $\wti{M}$ turns the cellular groups $C_i(\wti{M};\Z)$ into left-modules over the group ring $\Z[\pi]$. 

The representation $\alpha$ turns $S^k$ into a right $\Z[\pi]$-module.
We write
\[ C_*^\alpha(M;S^k):=S^k\otimes_{\Z[\pi]}C_*(\wti{M})\]
and we write
\[ H_i^\alpha(M;S^k):=H_i\big(C_*^\alpha(M;S^k)\big).\]
The  homomorphism $\psi$ turns $\Z[\G]$, and thus also $\RC_\xi \Z[\Gamma]$ into  a right $\Z[\pi]$-module. Thus we can view
$S[\Gamma]^k$ and $\RC_\xi S [\Gamma]^k=\RC_\xi \Z[\Gamma]\otimes_\Z S^k$ as right $\Z[\pi]$-module.
We write
\[ C_*^\alpha(M;\RC_\xi S[\Gamma]^k):=\RC_\xi S[\Gamma]^k\otimes_{\Z[\pi]}C_*(\wti{M})\]
and 
\[ H_i^\alpha(M;\RC_\xi S[\Gamma]^k):=H_i\big(C_*^\alpha(M;\RC_\xi S[\Gamma]^k)\big).\]

\bd 
Let $M$ be a topological space. We write $\pi=\pi_1(M)$. Let $(\psi\colon \pi\to \Gamma,\xi\in H^1(M;\R))$ be an admissible pair
and let
$\alpha\colon \pi_1(M)\to \gl(k,S)$ be a representation over a domain $S$.
The {\it $i$-th twisted Novikov-Betti number} is defined as
\[ b_i^\alpha(M,\psi,\xi):=\mbox{rank of the $\RC_\xi S[\Gamma]$-module $H_i^\alpha(M;\RC_\xi S[\Gamma]^k).$ }\]
The {\it $i$-th twisted Novikov torsion number} is defined as
\[ q_i(M,\psi,\xi):=\ba{c}\mbox{  minimal number of generators of the torsion}\\
\mbox{submodule of the  $\RC_\xi S[\Gamma]$-module $H_i^\alpha(M;\RC_\xi S[\Gamma]^k)$. }\ea\]
\ed

In the following proposition we list a few basic facts about twisted Novikov-Betti and torsion numbers. The proofs are verbatim the same as the proofs of the corresponding statements for untwisted Novikov-Betti and Novikov torsion numbers that are studied in \cite[Chapter~1]{F04}:

\bp\label{pr} The twisted Novikov-Betti and the twisted Novikov torsion numbers satisfy the following properties:
\bn
\item[(a)] The following equality holds: $$k\chi(M)=\sum_{i \geq 0} (-1)^i \cdot b_i^\alpha(M,\psi,\xi).$$
\item[(b)] For  any $\lambda \in \R_{>0}$ and any $i$ we have  
\[ H_i^\alpha(M;\RC_{\lambda\xi} S[\Gamma]^k)\,=\, H_i^\alpha(M;\RC_\xi S[\Gamma]^k).\]
In particular, we have 
\[ b_i^\alpha(M,\psi,\xi)=b_i^\alpha(M,\psi,\lambda\xi)\mbox { and } q_i^\alpha(M,\psi,\xi)=q_i^\alpha(M,\psi,\lambda\xi).\]
\en
\ep


\subsection{Topology of complex hypersurface complements}\label{top}

Let $X$ be a hypersurface in $\C^{n}$ ($n \geq 2$), with underlying reduced hypersurface $X_{red}$ defined by the (square-free) equation
$f=f_1\cdots f_s =0$, where $f_i$ are the irreducible factors of the polynomial 
$f$.  Let $X_i=\{f_i=0\}$ denote the irreducible components of
$X_{red}$. Embed $\C^{n}$ in $\cp^{n}$ by adding the hyperplane at infinity, $H$,
and let $\overline{X}$ be the projective completion of $X$ in $\cp^{n}$. Let $M_X$ denote the affine hypersurface complement 
$$M_X:=\C^n \setminus X= \C^n \setminus X_{red}.$$
Alternatively, $M_X$ can be regarded as the complement in $\cp^n$ of the
divisor $\overline{X} \cup H$. Then it is well-known that $H_1 (M_X;\Z)$ is
a free abelian group,  generated by the meridian loops $\gamma_i$ about the
non-singular part of each irreducible component ${X}_i$, for
$i=1,\cdots, s$ (e.g., see \cite{Di92}, (4.1.3), (4.1.4)). 
Furthermore, since $M_X$ is an $n$-dimensional affine variety, it has the
homotopy type of a finite CW-complex of real dimension $n$ (e.g., see  \cite{Di92}, (1.6.7), (1.6.8)). 

Let $S^{\infty}$ be a $(2n-1)$-sphere in $\C^{n}$ of a sufficiently large radius
(that is, the boundary of a small tubular neighborhood in $\cp^{n}$
of the hyperplane $H$ at infinity). Denote by
$$X^{\infty}=S^{\infty} \cap X$$ the {\it link of $X$ at infinity}, and by 
 $$M_X^{\infty}=S^{\infty} \sm X^{\infty}$$ its complement in $S^{\infty}$. Note that $M_X^{\infty}$ is homotopy equivalent to
$T(H) \setminus \overline{X} \cup H$, where $T(H)$ is the tubular neighborhood of $H$ in $\cp^{n}$ for which $S^{\infty}$ is the boundary.  Then a classical argument based on the Lefschetz hyperplane theorem yields that the homomorphism $$\pi_i(M_X^{\infty}) \lra \pi_i(M_X)$$ induced by inclusion is an isomorphism for $i < n-1$ and it is surjective for $i=n-1$; see \cite[Section 4.1]{DL06} for more details.  It follows that \be\label{eq}\pi_i(M_X,M_X^{\infty})=0 \ , \  {\rm for \ all} \  \ i \leq n-1,\ee
hence $M_X$ has the homotopy type of a complex obtained from $M_X^{\infty}$ by adding cells of dimension $\geq n$.

If, moreover, $X$ is {\it in general position at infinity}, that is, the reduced underlying variety of $\overline{X}$ is transversal to $H$ in the stratified sense, then $M_X^{\infty}$ is a circle fibration
over $H \setminus  \overline{X} \cap H$,  which is homotopy equivalent to the complement in $\C^{n}$ to the
affine cone over the projective hypersurface $\overline{X} \cap H \subset H=\mathbb{CP}^{n-1}$ (for a similar argument see \cite[Section~4.1]{DL06}). Hence, by the Milnor fibration theorem (e.g., see \cite[(3.1.9),(3.1.11)]{Di92}), $M_X^{\infty}$ fibers over $\C^* \simeq S^1$, with fiber $F$ homotopy equivalent to a finite $(n-1)$-dimensional CW-complex.


\section{Novikov homology of complex hypersurface complements}
\label{section:complex-hypersurface-complements}

In this section we will give the proof of Theorem~\ref{mt}.

\subsection{Preliminary lemmas}
We start out with the following lemma.

\begin{lemma}\label{lem:higher-homology}
Let $X \subset \C^n$ be a hypersurface  with complement $M_X$. Let  $(\psi\colon \pi_1(M_X)\to \Gamma,\xi \in H^1(M_X;\R))$ be an admissible pair and let  $\alpha\colon \pi_1(M_X)\to \gl(k,S)$ be a representation over a domain $S$. Then the following hold:
\begin{enumerate}
\item We have
 $H_i(M_X;\RC_\xi S[\Gamma]^k)=0$ for $i>n$. In particular, we have the vanishing $b_i^\a(M_X,\psi,\xi)=0$ and $q_i^\a(M_X,\psi,\xi)=0$ for $i>n$.
\item We have $q_n^\alpha(M_X,\psi,\xi)=0$.
\end{enumerate}
\end{lemma}

\begin{proof}
The first statement follows immediately from the fact that $M_X$ has the homotopy type of a finite CW-complex $M'$ of real dimension $n$.
This fact also implies that 
$ H_n^\a(M_X;\RC_\xi S[\Gamma]^k)=H_n^\a(M';\RC_\xi S[\Gamma]^k)$ is a submodule of the free $\RC_\xi S[\Gamma]$-module $C_n^\alpha(M';\RC_\xi S[\Gamma]^k)$.
In particular $ H_n^\a(M_X;\RC_\xi S[\Gamma]^k)$ has no $\RC_\xi S[\Gamma]$-torsion, which in turn implies that
 $q_n^\a(M_X,\psi,\xi)=0$.
\end{proof}

In the following we adopt the convention that if $\varphi\colon \pi_1(M)\to G$ is a homomorphism and $N\subset M$ is a connected subspace, then, by a slight abuse of notation, we denote the induced homomorphism $\pi_1(N)\to \pi_1(M)\xrightarrow{\varphi} G$ by $\varphi$ as well.

\bp\label{bound} 
 Let $X \subset \C^n$ be a  hypersurface  with complement $M_X$, and fix an admissible pair  $(\psi\colon \pi_1(M_X)\to \Gamma,\xi \in H^1(M_X;\R))$. Let  $\alpha\colon \pi_1(M_X)\to \gl(k,S)$ be a representation over a domain $S$. Then for any 
 $i < n-1$, we have $\RC_\xi S[\Gamma]$-isomorphisms
  \[
 H_i^\a(M_X^\infty;\RC_\xi S[\Gamma]^k)\xrightarrow{\cong} H_i(M_X;\RC_\xi S[\Gamma]^k)\]
and we  have an epimorphism of $\RC_\xi S[\Gamma]$-modules
  \[
  H_{n-1}(M_X^\infty;\RC_\xi S[\Gamma]^k)\to H_{n-1}(M_X;\RC_\xi S[\Gamma]^k).\]
\ep

\begin{proof}
This is an immediate consequence of the fact,  mentioned in Section~\ref{top}, that the complement  $M_X$ is obtained (up to homotopy) from $M_X^{\infty}$ by adding cells of dimension $\geq n$ and the fact that twisted homology groups are homotopy invariants.
\end{proof}

\begin{definition}
Given a manifold $M$ and a class $\xi \in H^1(M;\Z)=\Hom(\pi_1(M),\Z)$ we say that $\xi$ is {\it fibered} if there exists a bundle map $p\colon M\to S^1$ such that $p_*=\xi\colon \pi_1(M)\to \Z$.\end{definition}

\begin{proposition}\label{prop:fibered}
 Let $M$ be a manifold. Let  $(\psi\colon \pi_1(M)\to \Gamma,\xi \in H^1(M;\Z))$ be an admissible pair and let  $\alpha\colon \pi_1(M)\to \gl(k,S)$ be a representation over a domain $S$. If $\xi$ is fibered, then for any $i$ we have
 $H_i^\a(M;\RC_\xi S[\Gamma]^k)=0$.
\end{proposition}

This  proposition is well-known to the experts, it can be proved along the lines of \cite[Theorem~4.2]{GM05} or alternatively \cite{Ch03,GKM05,FK06,Fr14}. Since we could not find a result in the literature which gives precisely the statement desired we  sketch a proof.

\begin{proof}
In the following, given a manifold $X$ and a map $\varphi\colon X\to X$ we denote by $T(X,\varphi)=(X\times [0,1])/(x,0)\sim(\varphi(x),1)$ the corresponding mapping torus. We refer to the induced map $\pi_1(T(X,\varphi)) \to \pi_1([0,1]/0\sim 1)=\Z$ as the canonical homomorphism. 
We can identify the manifold $M$ with a mapping torus $T:=T(X,\varphi)$ such that $\xi\in \Hom(\pi_1(M),\Z)=\Hom(T,\Z)$ agrees with the canonical homomorphism. 
Following \cite[Section~3]{FK06}  there exists a Meyer--Vietoris sequence
\[ \ba{l}\dots \to H_i(X\times [0,1];S[\Gamma]^k)\otimes_{S[\Gamma]} \RC_\xi S[\Gamma] 
\overset{ \id-t \varphi_*}{\longrightarrow}
  H_i(X\times [0,1];S[\Gamma]^k)\otimes_{S[\Gamma]} \RC_\xi S[\Gamma] \\
  \hspace{2cm} \to  H_i(M;\RC_\xi S[\Gamma]^k)\to \dots\ea \]
where $t$ is an element with $\xi(t)=1$. All the maps 
$ \id-t \varphi_*$ are invertible over $\RC_\xi S[\Gamma]$.
It follows that the the homology groups $H_i(M;\RC_\xi S[\Gamma]^k)$ vanish.
\end{proof}

\subsection{Novikov homology for positive integral cohomology classes}\label{ad}
\bd 
Let $X \subset \C^n$ be a hypersurface  with complement $M_X$. A cohomology class $\xi \in H^1(M_X;\R)$ is called {\it positive} if the corresponding group homomorphism $\xi \colon \pi_1(M_X) \to \R$ takes strictly positive values on each positively oriented meridian generator $\gamma_i$ about an irreducible component of $X_{red}$.\ed

The following theorem takes care of Theorem~\ref{mainthm} for {\it integral} cohomology classes.

\bt\label{m}
 Let $X \subset \C^n$ be a hypersurface  with complement $M_X$. 
We assume that $X$  is in general position at infinity.
 Let  $(\psi\colon \pi_1(M_X)\to \Gamma,\xi \in H^1(M_X;\Z))$ be an admissible pair and let  $\alpha\colon \pi_1(M_X)\to \gl(k,S)$ be a representation over a domain $S$.  
 If $\xi$ is positive, then 
\be\label{B} b_i^\a(M_X,\psi,\xi)= \left\{ \ba{rl} 0, & \mbox{ for }i\ne n, \\ (-1)^{n}k\cdot \chi(M_X),
&\mbox{ for }i=n. \ea \right.\ee
In particular,  \be\label{E}(-1)^n \cdot \chi(M_X) \geq 0.\ee
Moreover
\be\label{tor} q_i^\a(M_X,\psi,\xi)=0, \  {\rm for \ all \ } i \geq 0.\ee
\et

\begin{proof} 
Since $M_X$ has the homotopy type of a finite CW complex, the homology groups 
$H_i^\a(M_X;\RC_\xi S[\Gamma]^k)$ are finitely generated $\RC_\xi S[\Gamma]$-modules.

Let $f$ be a square-free polynomial  defining $X_{red}$, the reduced hypersurface underlying $X$. We denote the factors of $f$ by $f_1,\dots,f_s$.
Let $\xi \in H^1(M_X;\Z)$ be a positive integral cohomology class, with $(n_1,\cdots,n_s) \in \N^s$ the vector of values of $\xi\colon \pi_1(M_X) \to \Z$ on the positive meridians $\gamma_i$, $i=1,\cdots,s$, about the irreducible components of $X_{red}$ corresponding to the factors $f_1,\dots,f_s$. We consider the polynomial $g={f_1}^{n_1}\cdots {f_s}^{n_s}$ on $\C^n$. Clearly, the underlying reduced hypersurface $\{g=0\}_{red}$ coincides with $X_{red}$ and, moreover, the homomorphism $g_*\colon \pi_1(M_X) \to \Z$ induced by $g$ coincides with $\xi$ (cf. \cite[p.76-77]{Di92}). By Section \ref{top}, 
the element $\xi=g_*\in H^1(M_X^\infty;\Z)=\Hom(\pi_1(M_X^\infty),\Z)$ is fibered. 
It follows from Proposition~\ref{prop:fibered} that  $H_i^\a(M_X^\infty;\RC_\xi S[\Gamma]^k)=0$ for  all $i$. 

The theorem now follows from the combination of Proposition~\ref{pr} (a), 
Lemma~\ref{lem:higher-homology} and Proposition~\ref{bound}.
\end{proof}

\br The statement about the  vanishing of the classical Novikov-Betti numbers in Theorem~\ref{m} 
has also been obtained implicitely in \cite{Ma06}  by using Alexander modules.
Furthermore it can be also derived by using the corresponding vanishing statement for the $L^2$-Betti numbers of such complements, see \cite[Theorem~1.1]{Ma14}. Indeed, it follows from \cite[Proposition~2.4]{FLM} that we have the identification:
\be b_i(M_X;\xi)=b_i^{(2)}(M_X,\xi\colon \pi_1(M_X) \to \im(\xi))\ee
between the Novikov-Betti numbers and the $L^2$-Betti numbers corresponding to $\xi$.
However, to our knowledge, Novikov torsion numbers do not have such interpretation in terms of $L^2$-invariants.
\er

\subsection{Novikov homology for positive real cohomology classes: The proof of Theorem~\ref{mainthm}}\label{pos}

\begin{definition}\mbox{}
\begin{enumerate}
\item A {\it lattice} in an $n$-dimensional real vector space $V$ is an additive subgroup $L$ of $V$ of rank $n$ such that $L$ generates $V$ as a real vector space.
\item Let $V$ be a vector space with lattice $L$.
\bn
\item  An {\it open integral half-space} of $V$ is a subset of the form $f^{-1}(\R_{>0})$ where $f\colon V\to \R$ is a homomorphism that takes integral values on $L$. 
\item  A {\it closed integral half-space} of $V$ is a subset of the form $f^{-1}(\R_{\geq 0})$ where $f\colon V\to \R$ is a homomorphism that takes integral values on $L$. 
\item The intersection of finitely many open and closed integral half-spaces is called an {\it integral cone}.
\item A finite union of integral cones is called an {\it integral subset} of $V$.
\en
\end{enumerate}
\end{definition}

The following elementary lemma summarizes some properties of integral subsets.

\begin{lemma}\label{lem:integral-subsets}
Let $V$ be a vector space together with a lattice $L$.
\bn
\item The complement of an integral subset is again an integral subset.
\item The intersection of finitely many  integral subsets is again an integral subset.
\item The union of finitely many  integral subsets is again an integral subset.
\item Any non-empty integral subset contains at least one lattice point. 
\en
\end{lemma}

Let $\Gamma$ be a free abelian group of rank $n$. In the following we  always view $\Hom(\Gamma,\R)$ as equipped with the lattice $\Hom(\Gamma,\Z)$. 
We can now formulate the following technical proposition that will be proved in the next section.

\begin{proposition}\label{mainprop}
Let $S$ be a domain, let $\Gamma$ be a free abelian group and let $C_*$ be a chain complex of finite free $S[\Gamma]$-modules.
Then 
\[ \big\{ \xi\in \Hom(\Gamma,\R)\,\big|\, H_*(\RC_\xi S[\Gamma]\otimes_{S[\Gamma]}C_*)=0\}\]
is an integral subset of $\Hom(\Gamma,\R)$.
\end{proposition}

The statement of Proposition~\ref{mainprop} is closely related to work of Pajitnov, see e.g.\ \cite[Theorem~2.2]{Pa90}\cite[Corollary~2.7]{Pa07}. But to the best of our knowledge the statement of Proposition~\ref{mainprop} cannot be found in the literature. More precisely, all results that we could found that have similar statements are dealing only with $\xi$'s in $\Hom(\Gamma,\R)$ that are monomorphisms.

Assuming Proposition~\ref{mainprop} we are now in a position to complete the proof of Theorem~\ref{mt}.

\begin{proof}[Proof of Theorem~\ref{mt}]
 Let $X \subset \C^n$ be a complex hypersurface in general position at infinity, with complement $M_X$. Let $\psi\colon \pi_1(M_X)\to \Gamma$ be an epimorphism onto a free abelian group $\Gamma$ and let 
$\alpha\colon \pi_1(M_X)\to \gl(k,S)$ be a representation over a domain $S$. 
Using the notations from Section \ref{top}, let $X^{\infty}$ be the link at infinity of $X$, with complement $M_X^{\infty}$. As per our convention, we also denote by $\psi$ and $\alpha$ the induced epimorphism $\pi_1(M_X^{\infty}) \to \Gamma$ and, respectively, the representation $\pi_1(M^{\infty}_X)\to \gl(k,S)$. Clearly, an admissible pair $(\psi,\xi)$ for $M_X$ induces an admissible pair for $M^{\infty}_X$.

 As in the proof of Theorem \ref{m}, it follows from Lemma~\ref{lem:higher-homology} and Proposition~\ref{bound} that
  it suffices to extend the vanishing $H_*^\a(M_X^\infty;\RC_\xi S[\Gamma]^k)=0$ to all positive real cohomology classes $\xi$, with $(\psi,\xi)$ admissible.

We denote by $\wti{M^{\infty}_X}$ the universal cover of $M^{\infty}_X$ and we write
\[ C_*:= S[\Gamma]^k\otimes_{\Z[\pi_1(M^{\infty}_X)]}C_*(\wti{M^{\infty}_X}).\]
In the following, given $\xi \in \Hom(\Gamma,\R)$ we denote the induced composite homomorphism $\pi_1(M^{\infty}_X)\to \pi_1(M_X) \to \Gamma\to \R$ by $\xi$ as well.
Note that for any $\xi\colon \Gamma\to \R$ we have
\[\ba{rcl} H_*\big(\RC_\xi S[\Gamma]\otimes_{S[\Gamma]}C_*\big)&=&
H_*\big(\RC_\xi S[\Gamma]\otimes_{S[\Gamma]}S[\Gamma]^k\otimes_{\Z[\pi_1(M^{\infty}_X)]}C_*(\wti{M^{\infty}_X})\big)\\[0.1cm]
&\cong & H_*\big(\RC_\xi S[\Gamma]^k\otimes_{\Z[\pi_1(M^{\infty}_X)]}C_*(\wti{M^{\infty}_X})\big) \\[0.1cm] 
&=&
H_*(M^{\infty}_X;\RC_\xi S[\Gamma]^k).\ea\]
Combining this observation with  Proposition~\ref{mainprop} and with Lemma~\ref{lem:integral-subsets} (1) we see that 
\[ V\,:=\,\big\{ \xi \in \Hom(\Gamma,\R)\,\big|\,\mbox{ there exists an $i$ with }
H_i(M^{\infty}_X;\RC_\xi S[\Gamma]^k)\ne 0\big\}\]
is an integral subset of $\Hom(\Gamma,\R)$.

Now let $\mu_1,\dots,\mu_s$ be the generators of $\Gamma$ that correspond to the meridians of the $s$ irreducible components of $X_{red}$. Recall that for an admissible pair $(\psi,\xi)$ we say $\xi \in \Hom(\Gamma,\R)$ is positive if $\xi(\mu_i)>0$ for $i=1,\dots,s$. We denote by $\Hom^+(\Gamma,\R)$ the set of all positive homomorphisms. 

Clearly $\Hom^+(\Gamma,\R)$ is an integral subset of $\Hom(\Gamma,\R)$. 
From Lemma~\ref{lem:integral-subsets} (2) we deduce that 
$\Hom^+(\Gamma,\R)\cap V$ is an integral subset. 
By Theorem~\ref{m} we know that $\Hom^+(\Gamma,\Z)\cap  V=\emptyset$. Put differently, $\Hom^+(\Gamma,\R)\cap V$ does not contain a lattice point.
It follows from Lemma~\ref{lem:integral-subsets} (4) that $\Hom^+(\Gamma,\R)\cap V=\emptyset$. But that means exactly that $H_*(M^{\infty}_X;\RC_\xi[\Gamma]^k)=0$ for all  $\xi\in \Hom^+(\Gamma,\R)$.
\end{proof}

\subsection{Proof of Proposition~\ref{mainprop}}
Before we can give the proof of Proposition~\ref{mainprop} we need to formulate two more lemmas.

\begin{definition}
Let $\Gamma$ be a free abelian group and let $S$ be a domain.
Given $p\in S[\Gamma]$ we write
\[ M(p)\,:=\,\{\xi\in \Hom(\Gamma,\R)\,|\, p\mbox{ is invertible in }\RC_\xi S[\Gamma]\}\]
and given a matrix $A$ over $S[\Gamma]$ we write
\[ M(A)\,:=\,\{\xi\in \Hom(\Gamma,\R)\,|\, A\mbox{ is invertible over }\RC_\xi S[\Gamma]\}.\]
Furthermore, given  a chain complex $C_*$ over $S[\Gamma]$ we write
\[ M(C_*)\,:=\,
\big\{ \xi\in \Hom(\Gamma,\R)\,\big|\, H_*(\RC_\xi S[\Gamma]\otimes_{S[\Gamma]}C_*)=0 \}.\]
\end{definition}

\begin{lemma}\label{lem:determinant-pcones}
Let $S$ be a domain  and let $\Gamma$ be a free abelian group.
For any $p\in S[\Gamma]$ and for any matrix $A$ over $S[\Gamma]$ the sets $M(p)$ and $M(A)$ are integral subsets of $\Hom(\Gamma,\R)$. 
\end{lemma}

\begin{proof}
Clearly it suffices to prove the lemma for any non-zero $p\in S[\Gamma]$.
By Lemma~\ref{lem:units-in-rc} we have
\[ M(p)=\{ \xi\in \Hom(\Gamma,\R)\,|\, t_\xi(p)\mbox{ is a monomial}\}.\]
We write $p=\sum_{i=1}^n a_ig_i$ with $a_1,\dots,a_n\in S\sm \{0\}$ and where $g_1,\dots,g_n$ are pairwise disjoint elements of $\Gamma$.  Furthermore we can arrange that there exists an $m$ such that $a_1,\dots,a_m$ are units in $S$ and such that $a_{m+1},\dots,a_n$ are not units in $S$. It follows from the above description of $M(p)$ that
\[ M(p) = \bigcup\limits_{i=1}^m
\big\{ \xi\in \Hom(\Gamma,\R)\,|\, \xi(a_i)>\xi(a_j)\mbox{ for all }j\ne i\big\}.\]
This shows that $M(p)$ is the disjoint  union of finitely many open integral cones, in particular $M(p)$ is an integral subset.
\end{proof}

\begin{definition}
Let $R$ be a domain.
\bn
\item We say that a chain complex $C_*$ of free $R$-modules is {\it based} if each chain module $C_i$ is equipped with a basis. 
\item Let $C_*$ be a based finite chain complex  of length $m$ of finitely generated free  $R$-modules. We denote by $A_i=(a_{jk}^i)$, $i=0,\dots,{m-1}$ the corresponding boundary matrices. (Here following the convention of \cite{Tu01}, we think of elements in $R^k$ as row vectors and we think of the matrices as multiplying on the right.)
Following \cite[p.~8]{Tu01} we define a {\it matrix chain for $C_*$} to be a collection of sets $\alpha=(\alpha_0,\dots,\alpha_m)$ where $\alpha_i\subset \{1,2,\dots,\dim C_i\}$ so that $\alpha_0=\emptyset$. We denote by $A_i(\alpha)$ the submatrix of $A_i$ formed by the entries $a_{jk}^i$ with $j\in \alpha_{i+1}$ and $k\not\in \alpha_i$. 
The matrix chain $\alpha$ is called a {\it $\tau$-chain} if $A_0(\alpha),\dots,A_{m-1}(\alpha)$ are square matrices. \en
\end{definition}

 The following lemma is precisely \cite[Lemma~2.5]{Tu01}.

\begin{lemma}\label{lem:vanishing-novikov-homology}
Let $R$ be a domain and let $C_*$ be a based finite chain complex  of finitely generated free  $R$-modules. We denote by $A_*$ the corresponding boundary matrices.
Then  $H_i(C_*)=0$  if and only if there exists a $\tau$-chain $\alpha$ such that $\det(A_i(\alpha))$ is invertible over $R$ for all $i$.
\end{lemma}

\begin{proof}[Proof of Proposition~\ref{mainprop}]
Let $S$ be a domain, let $\Gamma$ be a free abelian group and let $C_*$ be a chain complex of finite free $S[\Gamma]$-modules of length $m$. We pick a basis for each chain module $C_i$. 
We denote by $A_i$ the corresponding boundary matrices of the chain complex.
It follows from Lemma~\ref{lem:vanishing-novikov-homology} that
\[ M(C_*)=\bigcup\limits_\alpha \big\{ \xi \in \Hom(\Gamma,\R)\,|\,\mbox{  $ \det(A_i(\alpha))$ is invertible over $\RC_\xi S[\Gamma]$ for all $i$}\big\},\]
where we take the union over all $\tau$-chains. Put differently, we have
\[  M(C_*)=\bigcup\limits_\alpha \bigcap\limits_{i=0}^{k-1} M(A_i(\alpha)).\] 
Each $M(A_i(\alpha))$ is by Lemma~\ref{lem:determinant-pcones} an integral subset of $\Hom(\Gamma,\R)$. It follows from Lemma~\ref{lem:integral-subsets} (2) and (3)  that $M(C_*)$ is also an integral subset.
\end{proof}

\end{document}